\documentclass[10pt,a4paper,twoside]{article}
\usepackage{amssymb}
\usepackage{amsmath}
\usepackage{amsthm}
\usepackage{graphicx}

\def\di{\text{\rm di-}}
\def\tri{\text{\rm tri-}}
\def\pre{\text{\rm pre-}}
\def\post{\text{\rm post-}}
\def\As{\mathrm{As}}
\def\Alt{\mathrm{Alt}}
\def\Mal{\mathrm{Mal}}
\def\Jord{\mathrm{Jord}}
\def\Lie{\mathrm{Lie}}
\def\Leib{\mathrm{Leib}}
\def\Perm{\mathrm{Perm}}

\def\ComTrias{\mathrm{ComTrias}}
\def\Id{\mathrm{Id}}
\def\id{\mathrm{id}}

\theoremstyle{plain}
\newtheorem{theorem}{Theorem}[section]
\newtheorem{lemma}[theorem]{Lemma}
\newtheorem{proposition}[theorem]{Proposition}

\newtheorem{corollary}[theorem]{Corollary}
\theoremstyle{definition}
\newtheorem{definition}[theorem]{Definition}
\newtheorem{example}[theorem]{Example}
\theoremstyle{remark}
\newtheorem{remark}{Remark}

\let\medcirc\circ
\let\medbullet\bullet
\let\EuScript\mathcal

\begin{document}

 
\begin{center}
\vspace*{2pt}
{\Large \textbf{Operads of decorated trees and their duals}}\\[30pt]
{\large {\emph{Vsevolod Yu.~Gubarev, Pavel S.~Kolesnikov}}}
\\[30pt]
\end{center}
{\footnotesize\textbf{Abstract}  
This is an extended version of a talk presented by the second author on the
Third Mile High Conference on Nonassociative Mathematics (August 2013, Denver,
CO). The purpose of this paper is twofold. First, we would like to review the
technique developed in a series of papers for various classes of di-algebras and show how do the same
ideas work for tri-algebras. Second, we present a general approach to the
definition of pre- and post-algebras which turns out to be equivalent to the
construction of dendriform splitting. However, our approach is
more algebraic and thus provides simpler way to prove various properties of
pre- and post-algebras in general.
}


\section{Introduction}

The study of a wide variety of algebraic systems that may be informally called di-algebras was initiated 
by J.-L.~Loday and T.~Pirashvili 
\cite{LodayPir1993}, who proposed the notion of an (associative) di-algebra as a
tool in  cohomology theory of Lie and Leibniz algebras.
A systematic study of associative di-algebras and their Koszul dual dendriform algebras 
was presented in \cite{Loday01}.
Later, an algebraic approach to operads appearing in combinatorics led J.-L.~Loday and M.~Ronco
\cite{TriAndPostAs} to the notions of tri-associative and tri-dendriform algebras. 

In \cite{Chapoton01}, F.~Chapoton pointed out that the operads governing the varieties
of Leibniz algebras and of di-algebras in the sense of \cite{LodayPir1993} may be presented 
as Manin white products \cite{GinKapr94} of the operad $\Perm $ with $\Lie $ and $\As$, respectively. 
Manin products (white product~$\medcirc $ and black product~$\medbullet $) were originally defined for quadratic associative algebras 
and binary quadratic operads.  In \cite{Vallette08_posets}, it was proposed a conceptual approach 
to Manin products and Koszul duality which covers a wide range of monoids in categories with two coherent 
monoidal products. The operad $\Perm $ has an extremely simple algebraic nature, so it is obvious 
that the white product $\Perm \medcirc \mathfrak M$ coincides with the Hadamard product 
$\Perm \otimes \mathfrak M$ for every binary quadratic operad $\mathfrak M$ 
(see \cite{Vallette08_posets}). 
In this way, a general definition 
of a di-algebra over $\mathfrak M$ as an algebra governed by $\Perm\otimes \mathfrak M$
was considered in \cite{Kol2008a}, where it was shown that di-algebras are 
closely related with pseudo-algebras in the sense of \cite{BDK2001}.
This relation allowed solving many algebraic problems on di-algebras \cite{DiTKK,KolVoronin2013,Voronin2013},
and it is interesting to find an analogous construction for tri-algebras as well. 
It was also shown in \cite{Vallette08_posets} that the operad $\ComTrias $ (introduced in \cite{Vallette_2007})
has the same property as $\Perm$: $\ComTrias \medcirc \mathfrak M = \ComTrias \otimes \mathfrak M$.
In this paper, we show how to recover an ``ordinary'' algebra from a given $(\ComTrias \otimes \mathfrak M)$-algebra
and apply the result to solve a series of problems on tri-algebras. 

Roughly speaking, a passage from an operad $\mathfrak M$ governing a variety of ``ordinary'' 
algebras (associative, Lie, Jordan, Poisson, etc.) to the operad $\di\mathfrak M$ or $\tri\mathfrak M$
may be performed by ``decoration'' of planar trees presenting the operad $\mathfrak M$.
(For di-algebras, the procedure was proposed in \cite{Kol2008a}, for tri-algebras---in \cite{GubKol2013}
in the case of binary operations.) In this sense, to decorate a tree one has to emphasize one (for di-algebras)
or several (for tri-algebras) leaves and assume the composition (grafting) of trees to preserve 
the decoration (see Section \ref{sec:Replication} for details).

A similar unified approach to the definition of dendriform 
algebras comes naturally from the general concept of Manin black product \cite{Vallette08_posets}.
Namely, the class of associative dendriform di-algebras $\mathrm{Dend}$ \cite{Loday01}
is known to be governed by $\pre\Lie\medbullet \As$, where $\pre\Lie $ is the operad of 
left-symmetric algebras. Obviously (see \cite{GinKapr94}), $\mathrm{Dend} = (\Perm\medcirc \As)^!$
since $\Perm^! =\pre\Lie$, $\As^!=\As$.
The same duality between white and black Manin product holds in the general settings
\cite{Vallette08_posets}. So, the natural way to define a dendriform version of an 
$\mathfrak M$-algebra is to consider the operad $\pre\Lie\medbullet \mathfrak M$ 
or $\post\Lie \medbullet \mathfrak M$ (the operad $\post\Lie$ was introduced in \cite{Vallette_2007}).
The explicit description of the corresponding varieties of such systems in terms of defining identities 
was proposed in \cite{BaiGuoNi2012} (as  di-successor and tri-successor algebras) and in \cite{GubKol2013} 
as (di- and tri-dendriform algebras). 
A generalization of the first construction has recently been 
published in \cite{PeiBaiGuo2013}:
$\EuScript B$-($\EuScript A$-)Sp($\mathfrak M$)-algebras are defined for an arbitrary 
operad $\mathfrak M$.
In this paper, we state another simple procedure of ``dendriform splitting'' and prove that the classes of 
systems obtained 
(called pre- or post-algebras, respectively) coincide with those already introduced 
in \cite{BaiGuoNi2012,GubKol2013,PeiBaiGuo2013}.

The purpose of this paper is twofold. First, we would like to review the technique 
developed in \cite{Kol2008a}, \cite{GubKol2013}, and \cite{KolVoronin2013} for various classes of
 di-algebras
and show how do the same ideas work for tri-algebras. Second, we present a general 
approach to the definition of pre- and post-algebras which turns out to be equivalent 
to the construction of splitting proposed in \cite{PeiBaiGuo2013}. 
However, our approach is more algebraic and thus provides simpler way to prove 
various properties of pre- and post-algebras in general.

The paper is organized as follows. In Section \ref{sec:Replication} we recall 
the general definition of what is a di- or tri-algebra and explain 
its relation with averaging operators. 
Section \ref{sec:Splitting} is devoted to 
a construction generalizing Manin black products $\pre\Lie\medbullet \mathfrak M$
and $\post\Lie\medbullet \mathfrak M$ to an arbitrary (not necessarily binary or quadratic) 
operad $\mathfrak M$. The classes of $\pre\mathfrak{M}$- and $\post\mathfrak{M}$-algebras
obtained are closely related with Rota---Baxter operators in the very same way as 
($\EuScript A$-)Sp($\mathfrak M$)- and $\EuScript B$Sp($\mathfrak M$)-algebras
in \cite{PeiBaiGuo2013}, thus, our approach leads to the same classes of systems.
In Section \ref{sec4} we observe a series of algebraic problems related with di- and tri-algebras. 
Most of natural problems in this area may be easily reduced to similar problems in ``ordinary'' 
algebras by means of the embedding proved in Theorem~\ref{thm:EmbeddingIntoCurrent}.
 Section \ref{sec5} is devoted to analogous problems on pre- and post-algebras. 
In these classes, the picture is obscure: It is possible to state that many classical 
algebraic problems (like those stated is Section \ref{sec4}) make sense for pre- and post-algebras, 
but it is not clear how to solve them.

Throughout the paper we will use the following notations:
$\mathcal P(n)$ is the set of all nonempty subsets of $\{1,\dots, n\}$;
$S_n$ is the group of all permutations of $\{1,\dots, n\}$.
An operad $\mathfrak M$ is a collection of $S_n$-modules
$\mathfrak M(n)$, $n\ge 1$, equipped with associative and equivariant composition rule, see, e.g.,
\cite{LodayVallette}.

Given a language $\Sigma $ (a set of symbols of algebraic operations $f$ together with their arities $\nu(f)$), 
by a $\Sigma $-algebra we mean a linear space equipped with 
algebraic operations from $\Sigma $. The class of all $\Sigma$-algebras as well as the corresponding (free)
operad we denote by $\mathfrak F_\Sigma $. If $\mathfrak M$ is a quotient operad of $\mathfrak F_\Sigma$
and a $\Sigma$-algebra $A$ belongs to the variety governed by $\mathfrak M$ then 
we say $A$ to be an $\mathfrak M$-algebra. 
We will use the same symbol $\mathfrak M$ to denote the entire variety governed by operad $\mathfrak M$.

The free algebra in the variety of all $\mathfrak M$-algebras 
generated by a set $X$ we denote by $\mathfrak M\langle X\rangle $.

\subsection*{Acknowledgements}
The authors are very grateful to Bruno Vallette and to the anonymous referee
for valuable remarks and comments. 
The second author is pleased to acknowledge hospitality of the University of Denver 
during the Third Mile High Conference on Nonassociative Mathematics. 
This work was partially supported by RFBR 12-01-00329 and 12-01-33031.

\section{Replicated algebras}\label{sec:Replication}

 \subsection{Replication of a free operad}

 In this section we present an explanation of the idea underlying
 the transition from ``ordinary'' algebras to di- and tri-algebras
 and discuss why these constructions are the only possible ones
 in a certain context.

 Let us consider the free operad $\mathfrak F=\mathfrak F_\Sigma $
 generated by operations $\Sigma $.
 According to the natural graphical interpretation, the spaces $\mathfrak F(n)$, $n\ge 1$,
 are spanned by planar trees with enumerated leaves (variables) and labeled vertices
 (operations). 
For example, if $\Sigma =\{ (\cdot*\cdot), [\cdot,\cdot] \}$ consists of two
 binary operations
 then the term $[x_1,(x_4*x_3)]*[x_2,x_5]$ may be identified with
\begin{center}
\includegraphics{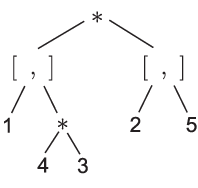}
\end{center}

 The general idea of replication (c.f. \cite{Kol2008a}) is to set an additional
 feature on the trees from $\mathfrak F(n)$: Emphasize one or several leaves
 and claim that the emphasizing is preserved by composition (grafting).
 Let us explain the details graphically and then present an equivalent algebraic
 statement.

 Recall the composition rule on the operad $\mathfrak F$: Given
 $T\in \mathfrak F(n)$, $T_i\in \mathfrak F(m_i)$, $i=1,\dots, n$,
 their composition $T(T_1,\dots, T_n) \in \mathfrak F(m_1+\dots + m_n)$
 is a tree obtained by attaching each $T_i$ to the $i$th leaf of $T$
 and by natural shift of numeration of leaves in each $T_i$. For example, if
 $\Sigma =\{ (\cdot*\cdot), [\cdot,\cdot] \}$ consists of two  binary operations,
$T=[x_2,x_3]*x_1$, $T_1=[x_2,[x_1,x_3]]$, $T_2=[x_2,x_1]$, $T_3=x_1*x_2$,
then $T(T_1,T_2,T_3)$ is presented by
\begin{center}
\includegraphics{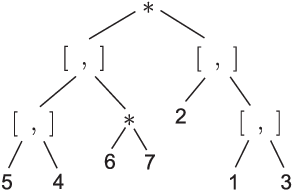}
\end{center}
 Symmetric group $S_n$ acts on $\mathfrak F(n)$ by permutations of leaves' numbers.

 By definition, every tree in $\mathfrak F(n)$ may be constructed by composition
 and symmetric group actions from the elementary trees (generators of the operad)
$f(x_1,\dots, x_n)$, $f\in \Sigma $, $\nu(f)=n$.

 Now, replace the generators by ``decorated'' elementary trees with one or several emphasized leaves
 and define the composition of such trees by the same rule as in $\mathfrak F$,
 assuming that: (1) attaching of a tree $T_i$ to a non-emphasized leaf of $T$ removes decoration from
 $T_i$; (2) attaching of a tree $T_i$ to an emphasized leaf of $T$ preserves decoration on $T_i$.
An example of such a composition with emphasized leaves circled in black is stated below.
\begin{center}
\includegraphics{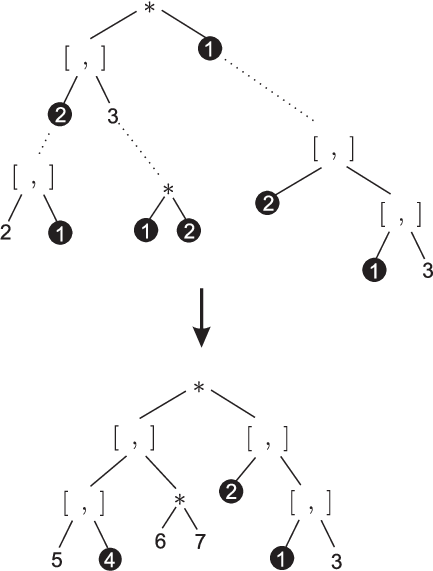}
\end{center}

 Note that if each of the trees $T,T_1,\dots, T_n$ has only one emphasized leaf then
 so is their composition $T(T_1,\dots, T_n)$. However, if we are allowed to emphasize
 more than one leaf (say, no more than two leaves of each tree, as in example above)
 then the composition may contain more emphasized leaves than
 each of the trees $T,T_1,\dots, T_n$ (see the example above).
 Hence, there are two natural cases: Either we may emphasize only one leaf of a tree
 (di-algebra case) or an arbitrary number of leaves (tri-algebra case).
 Let us denote the operads obtained by $\di\mathfrak{F}$ or $\tri\mathfrak{F}$,
 respectively.

  \subsection{Operads  Perm and ComTrias}
 Let us state definitions of two important operads.

\begin{example}[\cite{Chapoton01}]
Let $\Sigma $ contains one binary operation.
The operad governing the variety of associative algebras satisfying the identity $(x_1x_2)x_3 = (x_2x_1)x_3$
is denoted by $\Perm $. It is easy to see that monomials 
$e_i^{(n)}=(x_1\dots x_{i-1}x_{i+1}\dots x_n)x_i$,  $i=1,\dots, n$, 
form a linear basis of $\Perm(n)$, and thus
$\dim\Perm(n)=n$.
\end{example}

\begin{example}[\cite{Vallette_2007}] 
Given $n\ge 1$, let
$C(n)$ be the formal linear span of the set of ``corollas'' $\{e^{(n)}_H \mid H\in \mathcal P(n)\}$,
where $\mathcal P(n)$ stands for the collection
 of all nonempty subsets of $\{1,\dots, n\}$.
For $K\in \mathcal P(m)$,
$H_i\in \mathcal P(n_i)$, $i=1,\dots, m$, define the composition of sets
$K(H_1,\dots, H_m)\in \mathcal P(n_1+\dots + n_m)$ as follows:
\[
\begin{aligned}
 j\in K(H_1,\dots, H_m) \iff & \exists k\in K, l\in H_k: \\
 &  n_1+\dots + n_{k-1}<j\le n_1+\dots + n_k, \\
 & j=n_1+\dots + n_{k-1}+l.
\end{aligned}
\]
Then
\[
 e^{(m)}_K(e^{(n_1)}_{H_1}, \dots, e^{(n_m)}_{H_m}) = e^{(n)}_{K(H_1,\dots, H_m)},
\]
where $n=n_1+\dots + n_m$.

With respect to the natural action of the symmetric group, the family of spaces $C(n)$, $n\ge 1$, forms a
symmetric operad denoted by $\ComTrias$.
\end{example}

The algebraic interpretation of $\ComTrias$ was stated in
\cite{Vallette_2007}.
Namely, an algebra from the variety $\ComTrias $ is a linear space
equipped with two binary operations
$\perp $ and $\vdash $ satisfying the following axioms:
\[
\begin{gathered}
 (x\vdash y)\vdash z = x\vdash(y\vdash z), \quad
(x\vdash y)\vdash z = (y\vdash x)\vdash z,  \\
(x\perp y)\vdash z = (x\vdash y)\vdash z, \quad
x\vdash (y\perp z) = (x\vdash y)\perp z, \\
(x\perp y)\perp z = x\perp (y\perp z).
\end{gathered}
\]
It is easy to see that $e^{(n)}_H\in \ComTrias (n)$
may be identified with the monomial
\[
 x_{j_1}\vdash \dots \vdash x_{j_{n-k}}\vdash (x_{i_1}\perp \dots \perp x_{i_k}),
\]
where $H=\{i_1,\dots, i_k\}$, $i_1<\dots<i_k$, $j_1<\dots <j_{n-k}$.

\begin{example}\label{exmp:C_2}
Denote by $C_2$ a 2-dimensional space with a basis
$\{e_1,e_2\}$ and operations
\[
e_i\perp e_i = e_i, \quad e_1\vdash e_1 =  e_1,
\quad
e_1\vdash e_2 =  e_2,
\]
other products are zero. It is easy to check that $C_2\in \ComTrias$.
\end{example}

Note that the composition rule in the operad $\Perm $ is completely similar to the composition in $\ComTrias $
restricted to singletons: $e^{(n)}_i\in \Perm(n)$ may be identified with $e^{(n)}_{\{i\}}\in \ComTrias(n)$.

\begin{lemma}\label{lem:ComTriasCompSums}
 Let $m\ge 1$, $n_1,\dots, n_m\ge 1$, and let $n=n_1+\dots + n_m$.
Then
\begin{multline}\nonumber
 \sum\limits_{H\in \mathcal P(n)}\sum\limits_{\begin{subarray}{c} K, H_1,\dots, H_m \\
       K(H_1,\dots, H_m)=H \\
            K\in \mathcal P(m)\\
            H_i\in \mathcal P(n_i)
\end{subarray}}
 e^{(m)}_K\otimes e^{(n_1)}_{H_1}\otimes \dots \otimes e^{(n_m)}_{H_m} \\
=
\sum\limits_{ K\in \mathcal P(m)}\sum\limits_{ H_1\in \mathcal P(n_1)} \dots
\sum\limits_{ H_m\in \mathcal P(n_m)}
 e^{(m)}_K\otimes e^{(n_1)}_{H_1}\otimes \dots \otimes e^{(n_m)}_{H_m}.
\end{multline}
\end{lemma}

A similar statement holds for $\Perm $, if we restrict the sums to singletons only.

\begin{proof}
For $m=1$ the statement is obvious. It is enough to note that
\[
K(H_1,\dots, H_m)=
 \begin{cases}
   (K\setminus \{m\})(H_1,\dots, H_{m-1}) \cup (n-n_m + H_m), & m\in K, \\
   K(H_1,\dots, H_{m-1}), & m\notin K,
 \end{cases}
\]
and proceed by induction on $m$.
\end{proof}

 \subsection{Defining identities}

Let $\mathfrak M$ be a variety of $\Sigma $-algebras satisfying a family of polylinear identities $\Id(\mathfrak M)$.
Denote the operad governing this variety by the same symbol $\mathfrak{M}$, this is an image
of the free operad $\mathfrak F = \mathfrak F_\Sigma $ with respect to a morphism of operads whose
kernel equals $\Id(\mathfrak M)$.

\begin{definition}[\cite{Kol2008a,GubKol2013}]\label{defn:Di-Tri-Algebras}
Denote by $\di\mathfrak M$ and $\tri\mathfrak M$ the following Hadamard products of operads:
\[
 \di\mathfrak M = \Perm\otimes \mathfrak M , \quad
\tri\mathfrak M = \ComTrias\otimes \mathfrak M .
\]
\end{definition}

As an immediate corollary of this definition, we obtain

\begin{proposition}[\cite{Kol2008a, KolVoronin2013}]
 Let $A\in \mathfrak M $, $P\in \Perm $. Then
$P\otimes A$ equipped with operations
\[
\begin{gathered}
 f_i(x_1\otimes a_1, \dots , x_n\otimes a_n) = e_i^{(n)}(x_1, \dots ,x_n) \otimes f(a_1,\dots , a_n), \\
f\in \Sigma,\ \nu(f)=n,\ x_i\in P,\ a_i\in A,\ i=1,\dots, n,
\end{gathered}
\]
belongs to the variety $\di\mathfrak M$.
\end{proposition}

\begin{proposition}
 Let $A\in \mathfrak M $, $C\in \ComTrias $. Then
$C\otimes A$ equipped with operations
\[
\begin{gathered}
 f_H(x_1\otimes a_1, \dots , x_n\otimes a_n) = e_H^{(n)}(x_1, \dots ,x_n) \otimes f(a_1,\dots , a_n), \\
f\in \Sigma ,\ \nu(f)=n,  \  H\in \mathcal P (n),\ x_i\in C,\ a_i\in A,\ i=1,\dots, n,
\end{gathered}
\]
belongs to the variety $\tri\mathfrak M$.
\end{proposition}

In general, it is not clear which operations generate a Hadamard product of two operads (even if
the operads are binary).
However, operads $\mathfrak P=\Perm ,\ComTrias $ are good enough to allow finding
generators and defining relations of $\mathfrak P\otimes \mathfrak M$. 
In particular, if $\mathfrak M$ is a binary quadratic operad
then $\mathfrak P\otimes \mathfrak M = \mathfrak P\medcirc \mathfrak M$, where $\medcirc $ stands for the
Manin white product of operads. 
The purpose of this section is to present explicitly defining relations of $\tri\mathfrak{M}$
(for $\di\mathfrak M$, the algorithm was presented in \cite{DiJordTriple}, 
see also \cite{KolVoronin2013}).

First, let us note that the operad $\tri\mathfrak F$ is generated by
\[
 \Sigma^{(3)} = \{f^H \mid f\in \Sigma,\, \nu(f)=n,\, H\in \mathcal P(n) \}.
\]
Indeed, there exists a morphism of operads
$\iota: \mathfrak F_{\Sigma^{(3)}}\to \tri\mathfrak F$
sending $f^H$ to $e^{(n)}_H\otimes f$, $f\in \Sigma$, $\nu(f)=n$.
Therefore, every $D\in \tri\mathfrak M$ may be considered as a $\Sigma^{(3)}$-algebra.
Note that for every $f,g\in \Sigma $, $\nu(f)=n$, $\nu(g)=m$,
and for every $a_k,b_j\in D$ we have
\begin{multline}\label{eq:Zero-Identities}
 f^H(a_1,\dots , a_{i-1}, g^S(b_1,\dots, b_m), a_{i+1}, \dots, a_n) \\
=
 f^H(a_1,\dots , a_{i-1}, g^Q(b_1,\dots, b_m), a_{i+1}, \dots, a_n)
\end{multline}
for all $H\in \mathcal P(n)$, $S,Q\in \mathcal P(m)$ provided that $i\notin H$.
Indeed, by the definition of $\ComTrias$, the composition
\[
 e^{(n)}_H(\id,\dots, \underset{i}{e^{(m)}_S}, \dots, \id)
\]
does not depend on $S$ if $i\notin H$.

Moreover, each $\iota(m): \mathfrak F_{\Sigma^{(3)}}(m)\to \tri\mathfrak F(m)$, $m\ge 1$,
is surjective. The natural algorithm of constructing a canonical pre-image $\Phi^H \in \mathfrak F_{\Sigma^{(3)}}(m)$
of $e^{(m)}_H\otimes \Phi \in \tri\mathfrak F(m)$
with respect to $\iota(m)$ is stated in
\cite{GubKol2013} for binary case. In the general case, the algorithm remains the same:
Assume the pre-images are constructed for all terms of degree smaller than $m$.
For a monomial $u=u(x_1,\dots, x_m)\in \mathfrak F(m)$,
one may consider $e^{(m)}_H\otimes u\in \tri\mathfrak F(m)$ as a planar tree
with emphasized leaves $x_{i_1},\dots, x_{i_k}$, where $\{i_1,\dots, i_k\}=H$.
If $u=f(v_1,\dots, v_n)$, $f\in \Sigma $, $\nu(f)=n$, $v_i\in \mathfrak F(m_i)$,
then choose
$K=\{i\mid v_i \mbox{ contains }x_j,\, j\in H\}$ and set
$u^H = f^K(v_1^{H_1},\dots, v_n^{H_n})$,
where
\[
 H_i=\begin{cases}
          \{j\mid j\in H,\,x_j \mbox{ appears in } v_i\}, & i\in K, \\
          \{1\}, & i\notin K.
     \end{cases}
\]

Next, suppose $\Phi(x_1,\dots, x_m) \in \mathfrak F(m)$ is a polylinear identity on all algebras of
a variety $\mathfrak M$,
i.e., $\Phi $ belongs to the kernel of natural morphisms of operads $\tau_{\mathfrak M}: \mathfrak F\to \mathfrak M$.
Then $e^{(m)}_H\otimes \Phi $ belongs to the kernel of
$\id\otimes \tau_{\mathfrak M}: \tri\mathfrak F \to \tri\mathfrak M$.
Hence, $\Phi^H(x_1,\dots, x_m) \in \mathfrak F_{\Sigma^{(3)}}$
is an identity on all algebras in $\tri\mathfrak M$.

Suppose the variety $\mathfrak M$ is defined by a set of
polylinear identities $S(\mathfrak M)\subset \Id(\mathfrak M)$.
As we have shown above, every algebra in $\tri\mathfrak M$
may be considered as a $\Sigma^{(3)}$-algebra satisfying
the collection of identities  $S^{(3)}(\mathfrak M)$ that consists of
\eqref{eq:Zero-Identities} and
$\Phi^H(a_1,\dots, a_m)=0$ for all $\Phi \in S(\mathfrak M)\cap \mathfrak F(m)$, $H\in \mathcal P(m)$,
$m\ge 1$.

Let us prove that $S(\tri\mathfrak M)=S^{(3)}(\mathfrak M)$, i.e.,
every $\Sigma^{(3)}$-algebra satisfying
$S^{(3)}(\mathfrak M)$
is actually an algebra of the variety governed by $\tri\mathfrak M$.

\begin{theorem}\label{thm:EmbeddingIntoCurrent}
Suppose $\nu(f)\ge 2$ for all $f\in \Sigma $.
Then
every $\Sigma^{(3)}$-algebra satisfying $S^{(3)}(\mathfrak M)$
may be embedded into an appropriate algebra of the form $C\otimes A\in \tri\mathfrak M $,
where $C\in \ComTrias $, $A\in \mathfrak M$.
\end{theorem}

An analogous statement for $\di\mathfrak M $
was proved in \cite{KolVoronin2013}.

\begin{proof}
Given an algebra  $T\in \tri\mathfrak M$, denote by
$T_0\subseteq T$ the linear span of all
\[
  (e^{(n)}_H \otimes f)(a_1,\dots, a_n) - (e^{(n)}_K \otimes f)(a_1,\dots, a_n) ,
\]
$K,H\in \mathcal P(n)$, $a_i\in T$, $f\in \Sigma $, $\nu(f)=n$.
It follows from the definition of $\ComTrias $ that $T_0$ is an ideal in
$T$, and $\bar T = T/T_0$ may be considered as a $\Sigma $-algebra.
Moreover, the direct sum of linear spaces
\[
 \widetilde T = \bar T \oplus T
\]
turns into a $\Sigma$-algebra  with respect to
operations
\begin{equation}\label{eq:HatConstruction}
f(\bar a_1+b_1, \dots , \bar a_n + b_n)
 = \overline{ f^{K}(a_1, \dots, a_n) } +
 \sum\limits_{H\in \mathcal P(n)}
  f^H(c_1^H, \dots, c_n^H),
\end{equation}
($K$ is an arbitrary set in $\mathcal P(n)$)
$f\in \Sigma $, $\nu(f) = n$, $a_i,b_i\in T$, and
\[
 c_i^H = \begin{cases}
           a_i, & i\notin H, \\
           b_i , & i\in H.
         \end{cases}
\]

\begin{lemma}\label{lem:Tilde_in_Var}
$\widetilde T\in \mathfrak M $.
\end{lemma}

\begin{proof}
In \cite{GubKol2013}, this statement was proved in the binary case. The general case
is similar.

Suppose $\Phi (x_1,\dots, x_m)\in S(\mathfrak M)$. Then
\eqref{eq:HatConstruction} and \eqref{eq:Zero-Identities} imply
$\Phi(\bar a_1+b_1, \dots , \bar a_m + b_m) = 0$ for $a_i,b_j\in T$
by induction on the length of monomials.
\end{proof}

Recall the algebra $C_2\in \ComTrias $ from Example~\ref{exmp:C_2}.
Note that the map
$T \to C_2\otimes \widetilde T $,
given by
\[
 a\mapsto e_1\otimes \bar a + e_2\otimes a \in C_2\otimes \widetilde T,\quad a\in T,
\]
is a homomorphism of $\Sigma^{(3)}$-algebras.
Indeed, let $f\in \Sigma $, $\nu(f)=n$, $H\in \mathcal P(n)$,
$x_i = e_1\otimes \bar a_i + e_2\otimes a_i$, $a_i\in T$, $i=1,\dots, n$.
Then
\begin{multline}\nonumber
 (e^{(n)}_H \otimes f)(x_1, \dots, x_n)
=
 e^{(n)}_H(e_1,\dots, e_1) \otimes f(\bar a_1, \dots, \bar a_n) \\
=
 \sum\limits_{K\in \mathcal P(n)}
  e^{(n)}_H(e^K_{1},\dots, e^K_{n}) \otimes f(c^K_1, \dots, c_n^K),
\end{multline}
where
\[
 e_{k}^K = \begin{cases}
               e_1, & k\notin K, \\
               e_2, & k\in K,
             \end{cases}
\quad
c_i^K = \begin{cases}
               \bar a_i, & i\notin K, \\
               a_i, & i\in K.
             \end{cases}
\]
It is easy to note from the definition of $C_2$ that
  $ e^{(n)}_H(e^K_{1},\dots, e^K_{n}) \ne 0$
if and only if $K=H$ (in this case, the result is equal to $e_2$).
Hence,
\begin{multline} \nonumber
 (e^{(n)}_H \otimes f)(x_1, \dots, x_n)
=
 e_1 \otimes f(\bar a_1, \dots, \bar a_n) + e_2\otimes f(c^H_1, \dots, c_n^H) \\
=
 e_1\otimes \overline {f^H(a_1,\dots, a_n)} + e_2\otimes f^H(a_1,\dots, a_n).
\end{multline}
\end{proof}

\begin{remark}\label{rem:UnitaryOps}
Note that Theorem \ref{thm:EmbeddingIntoCurrent} remains valid for
 languages with unary operators $t\in \Sigma $, $\nu(t)=1$, provided that $S(\mathfrak M)$
 includes identities stating all these $t$ are endomorphisms or derivations with respect 
 to all $f\in \Sigma $, $\nu(f)>1$. In this case,
 $T_0$ is invariant with respect to $t$, and thus $\widetilde T$ exists.
\end{remark}

Therefore, if $T$ satisfies $S^{(3)}(\mathfrak M)$  then it is a subalgebra in
$C_2\otimes \widetilde T\in \tri\mathfrak M$, so $T\in \tri\mathfrak M$.

As it was shown in \cite{KolVoronin2013}, the variety governed by
$\di\mathfrak M=\Perm\otimes \mathfrak M$
may be represented as a variety of  $\Sigma^{(2)}$-algebras defined by
$S^{(2)}(\mathfrak M)$, where $\Sigma^{(2)}$ and $S^{(2)}(\mathfrak M)$
are obtained from $\Sigma $ and $S(\mathfrak M)$ in the same way
as $\Sigma^{(3)}$ and $S^{(3)}(\mathfrak M)$ provided that we consider
only singletons $H=\{i\}\in \mathcal P(n)$.

Examples include Leibniz algebras ($\di\Lie$) \cite{DiLie}, dialgebras ($\di\As$)
\cite{LodayPir1993}, semi-special quasi-Jordan algebras ($\di\Jord$)
\cite{Kol2008a,DiJord1,DiJord2}, Lie and Jordan triple di-systems (di-LTS \cite{DiLieTriple}
and di-JTS \cite{DiJordTriple}),
Malcev di-algebras ($\di\Mal$) \cite{DiMal},
dual pre-Poisson algebras (di-Pois) \cite{Aguiar2000},
triassociative algebras ($\tri\As $) \cite{TriAndPostAs}.

\begin{example}\label{exmp:TriLie}
Let us write down defining identities of $\tri\Lie$-algebras.
An algebra from $\tri\Lie$ is a linear space
with three binary operations $[\cdot \perp\cdot ]$,
$[\cdot \vdash\cdot ]$,  and $[\cdot \dashv\cdot ]$,
$[a \vdash b ] = -[b \dashv a ]$,
such that $[\cdot \perp \cdot ]$ is a Lie operation, $[\cdot \dashv\cdot ]$ satisfies (right) Leibniz identity,
and they satisfy the following axioms:
\begin{equation}\label{eq:TriLie}
\begin{gathered}{}
[x_1\perp [x_2\dashv x_3]]=[[x_1\dashv x_2]\perp x_3]+[x_2\perp [x_1\dashv x_3]], \\
[x_1\dashv [x_2\perp x_3]]=[x_1\dashv [x_2\dashv x_3]].
\end{gathered}
\end{equation}
Let us note that the first identity of \eqref{eq:TriLie}
appeared recently in \cite{Uchino2011}.
\end{example}

\begin{lemma}\label{lem:Functorial_Tilde}
If $\varphi:T\to T'$ is a homomorphism of $\tri\mathfrak N$-algebras then
$\tilde\varphi: \widetilde T\to \widetilde T'$ defined by
$\tilde \varphi (\bar a) = \overline{\varphi(a)}$, $\tilde \varphi(a)=\varphi(a)$, $a\in A$,
is a homomorphism of $\mathfrak N$-algebras.
\end{lemma}

\begin{proof}
It follows from the construction (see Theorem \ref{thm:EmbeddingIntoCurrent}) that
$\varphi(T_0)\subseteq T'_0$. Hence, $\tilde \varphi $ is a well-defined map, and it is
straightforward to check that $\varphi $ is a homomorphism of $\mathfrak N$-algebras.
\end{proof}

\subsection{Averaging operators}

Theorem \ref{thm:EmbeddingIntoCurrent} provides a powerful tool for solving
various problems for di- and tri-algebras (see Section~\ref{sec4}). Let us state here
an equivalent definition of $\tri\mathfrak M$ by means of averaging operators.

\begin{definition}
Suppose $A$ is a $\Sigma $-algebra. A linear map $t: A\to A$ is called
an {\em averaging operator} on $A$ if
\[
 f(ta_1,\dots , ta_n) = tf(ta_1,\dots, ta_{i-1}, a_i, ta_{i+1}, \dots , ta_n)
\]
for all $f\in \Sigma $, $\nu (f)=n$, $a_j\in A$, $i,j = 1,\dots, n$.

Let us call $t$ a {\em homomorphic averaging operator} if
\[
 f(ta_1,\dots , ta_n)= tf(a_1^H,\dots , a_n^H),
\]
where $H\in \mathcal P(n)$,
\begin{equation}\label{eq:a^H_Definition}
 a_i^H = \begin{cases}
           a_i, & i\in H, \\
            ta_i, & i\notin H.
         \end{cases}
\end{equation}
\end{definition}

Given a $\Sigma $-algebra $A$ equipped with a homomorphic averaging operator
$t$, denote by $A^{(t)}$ the following $\Sigma^{(3)}$-algebra:
\[
 f^H(a_1,\dots , a_n)= f(a_1^H,\dots , a_n^H),
\]
where $f\in \Sigma $, $\nu(f)=n$, $H\in \mathcal P(n)$, $a_i\in A$,
$a_i^H$ are given by \eqref{eq:a^H_Definition}.

If $t$ were an averaging operator on $A$ then the same rule defines
$\Sigma^{(2)}$-algebra $A^{(t)}$ provided that all $H$ are singletons.

\begin{theorem}\label{thm:AveragingOper}
Suppose $\nu(f)\ge 2$ for all $f\in \Sigma $.
\begin{enumerate}
\item If $A\in \mathfrak M$ and $t$ is an averaging operator on $A$
then $A^{(t)}$ is a $\di\mathfrak M$-algebra.

\item If $A\in \mathfrak M$ and $t$ is a homomorphic averaging operator on $A$
then $A^{(t)}$ is a $\tri\mathfrak M$-algebra.

\item Every $D\in \di\mathfrak M$ may be embedded into $A^{(t)}$
for an appropriate $A\in \mathfrak M$ with an averaging operator~$t$.

\item Every $T\in \tri\mathfrak M$ may be embedded into $A^{(t)}$
for an appropriate $A\in \mathfrak M$ with a homomorphic averaging operator~$t$.
\end{enumerate}
\end{theorem}

\begin{proof}
Let us show (2) and (4) since (1) and (3) are in fact restrictions of the statements on tri-algebras.

To prove (2), it is enough to note (by induction on $m$) that for every
\[
 \Phi=\Phi(x_1,\dots, x_m)\in \mathfrak F(m)
\]
and for every $H\in \mathcal P(m)$
we have
\[
 \Phi^H(a_1,\dots, a_m) = \Phi(a_1^H,\dots , a_m^H), \quad a_i\in A.
\]
Moreover,
\eqref{eq:Zero-Identities} also hold on $A^{(t)}$ by definition of~$t$.

Statement (4) follows from Theorem \ref{thm:EmbeddingIntoCurrent}:
$T$ is a subalgebra of $C_2\otimes \widetilde T$. Consider $A=\widetilde T=\bar T\oplus T$
and define
\[
 ta = \bar a, \quad t\bar a = \bar a, \quad a\in T.
\]
It is easy to see by definition of operations on $\widetilde T$ that
$t$ is indeed a homomorphic averaging operator on $A$, and
$T\subseteq A^{(t)}$ is a $\Sigma^{(3)}$-subalgebra.
\end{proof}

\section{Splitted algebras}\label{sec:Splitting}

In this section, we observe an approach to the procedure of splitting of an operad \cite{BaiGuoNi2012}
that leads to classes of objects in some sense  dual to di- and tri-algebras.

\subsection{Definition and examples}

As above, let $\mathfrak M$ be a variety of $\Sigma $-algebras defined by a family
of polylinear identities $S(\mathfrak M)$.

Suppose $T$ is a $\Sigma^{(3)}$-algebra, and let $C\in \ComTrias $.
Define the following $\Sigma $-algebra structure on the space $C\otimes T$:
\begin{gather}\label{eq:ComTriasOtimes}
  f(a_1\otimes u_1, \dots , a_n\otimes u_n)
   =
  \sum\limits_{H\in \mathcal P(n)}
  e^{(n)}_{H} (a_1,\dots, a_n) \otimes f^{H}(u_1,\dots, u_n), \\
  f\in \Sigma , \ \nu(f)=n.
\end{gather}
Denote the $\Sigma$-algebra obtained by $C\boxtimes T$.

In a similar way (considering only singletons in \eqref{eq:ComTriasOtimes}) one may
define $P\boxtimes D$ for a $\Sigma^{(2)}$-algebra $D$ and $P\in \Perm $.

\begin{definition}\label{defn:postAlg_via_ComTrias}
A class of $\Sigma^{(2)}$-algebras $D$ such that
$P\boxtimes D\in \mathfrak M $ for all $P\in \Perm $
is denoted by $\pre\mathfrak M$.

A class of $\Sigma^{(3)}$-algebras $T$ such that
$C\boxtimes T\in \mathfrak M $  for all $C\in \ComTrias$
is denoted by $\post\mathfrak M$.
\end{definition}

It  is enough to check $P\boxtimes D$ and $C\boxtimes T$
for free algebras $P=\Perm\langle X \rangle $ and $C=\ComTrias\langle X\rangle $,
where $X=\{x_1,x_2,\dots \}$ is a countable set of symbols.

It is obvious that $\pre\mathfrak M$ and $\post\mathfrak M$ are varieties of
$\Sigma^{(2)}$- and $\Sigma^{(3)}$-algebras, respectively. Indeed, it is easy
to find their defining identities by the very definition.

\begin{example}\label{exmp:LSA}
 Suppose $\Sigma $ consists of one binary operation $[\cdot,\cdot]$, and let $\mathfrak M = \Lie $.
Then $\Sigma^{(2)}$ consists of two operations, say,
$[\cdot \vdash \cdot]$ and $[\cdot \dashv \cdot]$. A $\Sigma^{(2)}$-algebra $D$ belongs to $\pre\Lie $
if and only if
$\Perm\langle X\rangle \boxtimes D \in \Lie$, i.e.,
\[
 [(x_1\otimes a_1),(x_2\otimes a_2)] = x_1x_2 \otimes [a_1\vdash a_2] + x_2x_1\otimes [a_1\dashv a_2]
\]
is anti-commutative and satisfies the Jacobi identity. The anti-commutativity implies
\[
 [a_1\vdash a_2] =- [a_2\dashv a_1],  \quad a_1,a_2\in D.
\]
Denote
$[a\vdash b]$ by $ab$.
Let us check the Jacobi identity:
\begin{multline}
 [[(x_1\otimes a_1),(x_2\otimes a_2)],(x_3\otimes a_3)]   \\
= x_1x_2x_3 \otimes (a_1a_2)a_3 - x_3x_1x_2 \otimes a_3(a_1a_2)
 - x_2x_1x_3\otimes (a_2a_1)a_3  + x_3x_2x_1 \otimes a_3(a_2a_1) \\
=e_3^{(3)}\otimes ((a_1a_2)a_3 - (a_2a_1)a_3 )
 -  e_2^{(3)} \otimes a_3(a_1a_2) + e_1^{(3)} \otimes a_3(a_2a_1).
\end{multline}
Hence,
\begin{multline}
 [[(x_1\otimes a_1),(x_2\otimes a_2)],(x_3\otimes a_3)]
+ [[(x_2\otimes a_2),(x_3\otimes a_3)],(x_1\otimes a_1)]  \\
+ [[(x_3\otimes a_3),(x_1\otimes a_1)],(x_2\otimes a_2)]
=
e_1^{(3)}(a_3(a_2a_1) - (a_3a_2)a_1 + (a_2a_3)a_1  - a_2(a_3a_1))  \\
+
e_2^{(3)}((a_3a_1)a_2 -a_3(a_1a_2) + a_1(a_3a_2)  - (a_1a_3)a_2) \\
+
e_3^{(3)}((a_1a_2)a_3 - a_1(a_2a_3) + a_2(a_1a_3) - (a_2a_1)a_3  ).
\end{multline}
Hence, $D\in \pre\Lie$ if and only if the product $ab$ is left-symmetric.
\end{example}

Other well-known examples include pre-associative (dendriform) \cite{Loday01},
post-associa\-tive (tridendriform) \cite{TriAndPostAs}, pre-Poisson \cite{Aguiar2000},
pre-Jordan \cite{Pre-Jordan} algebras, as well as pre-Lie triple systems \cite{Pre-LieTriple}.

\subsection{Equivalent description}

Suppose $T$ is a $\Sigma^{(3)}$-algebra.
Denote by $\widehat T$ the direct sum of two isomorphic copies
of $T$ as of linear space:
\[
\widehat T=T\oplus T'.
\]
Assume the isomorphism is given by the correspondence
$a\leftrightarrow a'$, $a\in T$, and define
\begin{equation}\label{eq:HatOperations}
 f(a_1+b'_1, \dots , a_n+b'_n) =
\sum\limits_{H\in \mathcal P(n)} f^H(a_1,\dots, a_n)
+
 \left (\sum\limits_{H\in \mathcal P(n)}
  f^{H} (c^H_1,\dots, c^H_n) \right )',
\end{equation}
where $f\in \Sigma $, $\nu(f)=n$, and
\[
 c^H_i= \begin{cases}
   a_i, & i\notin H, \\
   b_i, & i\in H.
 \end{cases}
\]
Thus, $\widehat T$ carries the structure of a $\Sigma $-algebra.
For a $\Sigma^{(2)}$-algebra $D$, one may define $\widehat D$ in a similar
way assuming $f^H(x_1,\dots, x_n) = 0$ for $|H|>1$.

\begin{theorem}[c.f. \cite{GubKol2013}]\label{lem:Splitting_vs_RB}
The following statements are equivalent:
\begin{enumerate}
 \item $T\in \post\mathfrak M$;
 \item $\widehat T\in \mathfrak M $.
\end{enumerate}
\end{theorem}

Similarly, a $\Sigma^{(2)}$-algebra $D$ belongs to $\pre\mathfrak M$ if and only if $\widehat D\in \mathfrak M$.

\begin{proof}
Let us fix $C=\ComTrias\langle Y \rangle $, $Y$ is an infinite set. It is enough to prove
that (2) is equivalent to $C\boxtimes T\in \mathfrak M$.

Suppose $\Phi=\Phi (x_1, \dots, x_n) \in \mathfrak F(n)$ is a polylinear term of degree $n$
in the language~$\Sigma $.
Evaluate the term $\Phi $ in
$C\boxtimes T$:
\[
 \Phi(y_1\otimes a_1, \dots, y_n\otimes a_n) =
 \sum\limits_{H\in \mathcal P(n)}
  e^{(n)}_{H} (y_1,\dots, y_n) \otimes \Phi_{(H)}(a_1,\dots, a_n).
\]
This equation defines a family of $n$-linear functions
$\Phi_{(H)}:T^{\otimes n} \to T$, $H\in \mathcal P(n)$.

\begin{lemma}\label{lem:EquivInductionTerms}
In the algebra $\widehat T$, the following equations hold for $a_i\in T\subset \widehat T$
($i=1,\dots,n$):
\begin{gather}
\Phi_{(H)} (a_1,\dots, a_n)' = \Phi (d_1^H, \dots, d_n^H),
                         \label{eq:TermalCorrespondence} \\
 \sum\limits_{H\in \mathcal P(n)} \Phi_{(H)} (a_1,\dots, a_n) = \Phi (a_1,\dots, a_n).
                         \label{eq:TermalCorrespondence-2}
\end{gather}
where
\[
 d_i^H = \begin{cases}
          a_i', & i\in H, \\
          a_i, & i\notin H.
         \end{cases}
\]
\end{lemma}

\begin{proof}
If $n=1$ then
\eqref{eq:TermalCorrespondence} is trivial.
Proceed by induction on $n$. Assume
\[
 \Phi = f(\Psi_1, \dots, \Psi_m), \quad f\in \Sigma, \ \nu(f)=m,
\]
where $\Psi_i\in \mathfrak F(n_i)$,
$n_1+\dots + n_m = n$.
Suppose
$z_{ij}\in Y$ are pairwise different,
$a_{ij}\in T$,
$i=1,\dots , m$, $j=1,\dots, n_i$.
To simplify notations, denote
\[
 \bar z_i = (z_{i1}, \dots, z_{in_i}), \quad
 \bar a_i = (a_{i1}, \dots, a_{in_i}), \quad i=1,\dots, m.
\]
For $H_i\in \mathcal P(n_i)$,
denote by $\bar a_i^{H_i}$ the $n_i$-tuple
$(d_{i1}^{H_i}, \dots, d_{in_i}^{H_i})$
obtained from the initial one by ``adding primes'' to all those components that belong
to $H_i$.

Then
\begin{multline}\nonumber
 f(z_{11}\otimes a_{11},\dots, z_{mn_m}\otimes a_{mn_m}) \\
 =
 \sum\limits_{\begin{subarray}{c}
               K\in \mathcal P(m)\\
              H_1\in \mathcal P(n_1) \\
                  \dots \\
           H_m\in \mathcal P(n_m)
\end{subarray}}
 e^{(m)}_K (e^{(n_1)}_{H_1}(\bar z_1), \dots,   e^{(n_m)}_{H_m}(\bar z_m)  )
 \otimes
 f^K(\Psi_{1(H_1)}(\bar a_1), \dots , \Psi_{m(H_m)}(\bar a_m)  ) \\
 =
 \sum\limits_{\begin{subarray}{c}
               K\in \mathcal P(m)\\
              H_1\in \mathcal P(n_1) \\
                  \dots \\
           H_m\in \mathcal P(n_m)
\end{subarray}}
 e^{(n)}_{K(H_1,\dots, H_m)}(z_{11}, \dots , z_{mn_m})
 \otimes
 f^K(\Psi_{1(H_1)}(\bar a_1), \dots , \Psi_{m(H_m)}(\bar a_m)  ),
\end{multline}
where
$K(H_1,\dots, H_m)$ is the composition of sets from the definition of $\ComTrias $.

Hence, for every $H\in \mathcal P(n)$ we have
\begin{equation}\label{eq:T_Hexpression}
 \Phi_{(H)}(a_{11}, \dots , a_{mn_m})
 =\sum\limits_{\begin{subarray}{c} K, H_1,\dots, H_m \\ K(H_1,\dots, H_m)=H \end{subarray}}
 f^K(\Psi_{1(H_1)}(\bar a_1), \dots , \Psi_{m(H_m)}(\bar a_m)  ).
\end{equation}
By definition, every $H$ uniquely determines $K$ and $H_i$ for $i\in K$. Other $H_j$ (for $j\notin K$)
in \eqref{eq:T_Hexpression} run through the entire $\mathcal P(n_j)$.
Therefore,
\[
\Phi_{(H)}(a_{11}, \dots , a_{mn_m})'
 =
f^K(b_1, \dots, b_m)',
\quad
 b_i = \begin{cases}
 \Psi_{i(H_i)}(\bar a_i)', & i\in K, \\
 \sum\limits_{H_i\in \mathcal P(n_i)} \Psi_{i(H_i)}(\bar a_i)  , & i\notin K.
 \end{cases}
\]
By the inductive assumption,
\[
 \Psi_{i(H_i)}(\bar a_i)' = \Psi_i(\bar a_i^{H_i}),
\quad
 \sum\limits_{H_i\in \mathcal P(n_i)} \Psi_{i(H_i)}(\bar a_i)   = \Psi_i(\bar a_i).
\]
It remains to apply the definition of operations in $\widehat D$ \eqref{eq:HatOperations}
to prove \eqref{eq:TermalCorrespondence}.

To complete the proof,
apply \eqref{eq:T_Hexpression} and Lemma~\ref{lem:ComTriasCompSums}:
\begin{multline}\nonumber
 \sum\limits_{H\in \mathcal P(n)}\Phi_{(H)}(a_{11}, \dots , a_{mn_m})  \\
 =\sum\limits_{H\in \mathcal P(n)}\sum\limits_{\begin{subarray}{c} K, H_1,\dots, H_m \\ K(H_1,\dots, H_m)=H \end{subarray}}
 f^K(\Psi_{1(H_1)}(\bar a_1), \dots , \Psi_{m(H_m)}(\bar a_m)  )  \\
=
\sum\limits_{K\in \mathcal P(m)}
\sum\limits_{H_1\in \mathcal P(n_1)}\dots \sum\limits_{H_m\in \mathcal P(n_m)}
f^K(\Psi_{1(H_1)}(\bar a_1), \dots , \Psi_{m(H_m)}(\bar a_m)  ).
\end{multline}
Now \eqref{eq:TermalCorrespondence-2} follows from polylinearity of $f^K$ and inductive assumption.
\end{proof}

Let us finish the proof of the theorem.
If
$\widehat T\in \mathfrak M $ then  $\ComTrias(Y)\boxtimes T$
satisfies all defining identities of the variety $\mathfrak M $
by Lemma~\ref{lem:EquivInductionTerms}.

The converse is even simpler. Note that
$ \widehat T = C_2\boxtimes T$,
where $C_2$ is the 2-dimensional $\ComTrias$-algebra from Example~\ref{exmp:C_2}.
By the very definition, $\widehat T\in \mathfrak M$.
\end{proof}

\begin{remark}
 Note that the base field itself is a 1-dimensional algebra in $\ComTrias$. Therefore,
 if $A\in \pre\mathfrak M$ or $A\in \post\mathfrak M$ then
 $\Bbbk \boxtimes A \in \mathfrak M$. This observation explains the term ``splitting'':
 An operation $f\in \Sigma $, $\nu(f)=n$, splits into $n$ or $2^n-1$
 operations, $f = \sum\limits_H f^H$.
\end{remark}

\subsection{Rota-Baxter operators}
Let $A$ be a $\Sigma$-algebra.

\begin{definition}[c.f. \cite{BaiGuoNi2012}]
 A linear map $\tau : A\to A$ is said to be a Rota---Baxter operator of weight $\lambda $ ($\lambda \in \Bbbk $)
if
\begin{gather}
f(\tau(a_1),\dots, \tau (a_n)) = \sum\limits_{H\in \mathcal P(n)} \lambda^{|H|-1}\tau(f(a_1^H, \dots, a_n^H)),
  \label{eq:Rota-BaxterGeneral}  \\
a_i^H = \begin{cases}
               a_i, & i\in H, \\ \tau(a_i), & i\notin H,
              \end{cases} \label{eq:a^H_iDefn}
\end{gather}
for all $f\in \Sigma $, $\nu(f)=n$, $a_i\in A$.
\end{definition}

Obviously, if $\tau $ is a Rota---Baxter operator of nonzero weight $\lambda $ then
$\tau' = \frac{1}{\lambda }\tau $ is a Rota---Baxter operator of weight 1. Hence,
there are two essentially different cases: $\lambda = 0$ (zero weight) and $\lambda =1$ (unit weight).

The following statement was proved in \cite{GubKol2013} in the case of binary operations ($\nu(f)=2$).
By means of the approach presented in this paper, the proof becomes clear in the general case.

Given a  $\Sigma $-algebra $A$ equipped with
 a Rota---Baxter operator $\tau $, denote by $A^{(\tau )}$
the $\Sigma^{(3)}$-algebra defined on the space $A$ by
\[
 f^H(a_1,\dots, a_n) = f(a_1^H,\dots, a_n^H),
\]
where
$f\in \Sigma$, $i=1,\dots, n$, $a_1,\dots, a_n \in A$,
$a_i^H$ are given by \eqref{eq:a^H_iDefn}.

The same relations restricted to $|H|=1$ define a $\Sigma^{(2)}$-algebra structure on $A$
also denoted by $A^{(\tau)}$.

\begin{theorem}\label{thm:RB-operators}
 \begin{enumerate}
\item If $A\in \mathfrak M$ and $\tau $ is a Rota---Baxter operator of zero weight on $A$
then $A^{(\tau )}$ is a $\pre\mathfrak M$-algebra.

\item If $A\in \mathfrak M$ and $\tau $ is a Rota---Baxter operator of unit weight on $A$
then $A^{(\tau )}$ is a $\post\mathfrak M$-algebra.

\item Every $D\in \pre\mathfrak M$ may be embedded into $A^{(\tau)}$
for an appropriate $A\in \mathfrak M$ equipped with Rota---Baxter operator $\tau $ of zero weight.

\item Every $T\in \post\mathfrak M$ may be embedded into $A^{(\tau)}$
for an appropriate $A\in \mathfrak M$ equipped with Rota---Baxter operator $\tau $ of unit weight.
\end{enumerate}
\end{theorem}

\begin{proof}
 As in Theorem \ref{thm:AveragingOper}, let us prove (2) and (4).

For (2), it is enough to consider $C\boxtimes A^{(\tau )}$ for any $C\in \ComTrias$, and note
(by induction on $m\ge 1$) that
\[
 \Phi(y_1\otimes a_1,\dots, y_m\otimes a_m)
 = \sum\limits_{H\in \mathcal P(m)} e^{(m)}(y_1,\dots, y_m)\otimes \Phi (a_1^H,\dots, a_m^H)
\]
for every $\Phi \in \mathfrak F(m)$.

Hence, $C\boxtimes A^{(\tau )}\in \mathfrak M$.

To prove (4), consider $A=C_2\boxtimes T \in \mathfrak M$, where $C_2$ is the algebra from Example \ref{exmp:C_2},
 and define
\begin{equation}\label{eq:RB_Hat}
\tau (e_1\otimes a) = -e_1\otimes a, \quad \tau(e_2\otimes a)= e_1\otimes a, \quad a\in T.
\end{equation}
Let us show that \eqref{eq:RB_Hat} is a Rota---Baxter operator of unit weight
on $C_2\boxtimes T$. Indeed, suppose $f\in \Sigma $, $\nu(f)=n$,
$u_i=e_{k_i}\otimes a_i$, $k_i\in \{1,2\}$, $a_i\in T$, $i=1,\dots, n$.
Evaluate the left-hand side of \eqref{eq:Rota-BaxterGeneral}:
\[
 f(\tau(u_1), \dots , \tau(u_n)) = (-1)^{|K|}f(e_1\otimes a_1,\dots, e_1\otimes a_n),
\]
where $K=\{i\mid k_i=1\}$.
On the other hand,
\begin{multline}\nonumber
 f(u_1^H, \dots, u_n^H) =
 (-1)^{ |K\setminus H| } f(e_{k_1'}\otimes a_1,\dots, e_{k_n'} \otimes a_n) \\
= (-1)^{ |K\setminus H| } \sum\limits_{M\in \mathcal P(n)}
 e_M^{(n)} (e_{k_1'},\dots , e_{k_n'})\otimes f^M(a_1,\dots, a_n),
\end{multline}
where $k'_i=\begin{cases}
             k_i, & i\in H, \\
             1, & i\notin H
            \end{cases}$.
Nonzero summands appear in two cases: (1)  $k'_i=1$ for all $i=1,\dots, n$; (2) $k'_i=2$ if and only if $i\in M$.
The first case occurs if and only if $H\subseteq K$, the second one corresponds to $M=H\setminus K$.
Hence,
\[
 f(u_1^H, \dots, u_n^H)= \begin{cases}
 (-1)^{|K|-|H|} e_1 \otimes \sum\limits_{M\in \mathcal P(n)} f^M(a_1,\dots, a_n), & H\subseteq K, \\
 (-1)^{|K\setminus H|} e_2\otimes f^{H\setminus K}(a_1,\dots, a_n), & H\not\subseteq K.
\end{cases}
\]
Let us evaluate the right-hand side of \eqref{eq:Rota-BaxterGeneral}:
\begin{multline}\label{eq:RB_relation_check}
 \sum\limits_{H\in \mathcal P(n)}\tau (f(u_1^H, \dots, u_n^H)) 
=
\sum\limits_{\emptyset\ne H\subseteq K}(-1)^{|K|-|H|+1} f(e_1\otimes a_1,\dots , e_1\otimes a_n) \\
+
\sum\limits_{H\not\subseteq K} (-1)^{|K\setminus H|} e_1\otimes f^{H\setminus K}(a_1,\dots, a_n).
\end{multline}
The first summand in the right-hand side of \eqref{eq:RB_relation_check}
is equal to $(-1)^{|K|}f(e_1\otimes a_1,\dots , e_1\otimes a_n)$
since
\[
 \sum\limits_{H\subseteq K}(-1)^{|H|} = 1 + \sum\limits_{\emptyset\ne H\subseteq K}(-1)^{|H|} = 0.
\]
In the second summand, present $H\not\subseteq K$ as
$H=U\cup M$, $U\subseteq K$, $M\ne \emptyset$, $M\cap K=\emptyset$. Then
\[
\sum\limits_{U\subseteq K}\sum\limits_{\begin{subarray}{c} M\ne \emptyset \\ M\cap K=\emptyset\end{subarray}}
 (-1)^{|K|-|U|} e_1\otimes f^{M}(a_1,\dots, a_n) = 0
\]
by the same reasons.

We have proved that \eqref{eq:Rota-BaxterGeneral} holds for $\tau $ ($\lambda =1$), and thus it is a Rota---Baxter operator of unit weight.
\end{proof}

\begin{remark}
Theorem \ref{thm:RB-operators} implies that Definition \ref{defn:postAlg_via_ComTrias}
provides an equivalent description of the same class of systems as
the splitting procedure described in
in \cite{BaiGuoNi2012}: $\pre\mathfrak M = \EuScript{A}\mathrm{Sp}(\mathfrak M)$,
$\post\mathfrak M = \EuScript{B}\mathrm{Sp}(\mathfrak M)$.

In the binary case, $\pre\mathfrak M$ and $\post\mathfrak M$ coincide with
operads denoted in \cite{GubKol2013} by $\mathrm{DendDi}\mathfrak{M}$ and $\mathrm{DendTri}\mathfrak{M}$, respectively.
\end{remark}

\begin{remark}
Indeed, it was shown in \cite{GubKol2013} that
if $\mathfrak M$ is a binary quadratic operad then
$\pre\mathfrak M = \pre\Lie \medbullet \mathfrak M$,
$\post\mathfrak M = \post\Lie \medbullet \mathfrak M$,
where $\medbullet $ is the Manin black product of operads \cite{GinKapr94},
\[
(\pre\mathfrak M)^! = \di(\mathfrak M^!), \quad  (\post\mathfrak M)^! = \tri(\mathfrak M^!),
\]
where $!$ stands for Koszul duality of operads.
\end{remark}

\section{Problems on replicated algebras}\label{sec4}

In this section, we consider a series of problems for replicated algebras.
Some of them have already been solved in particular cases. Here we will show
how to solve them in general.

\subsection{Codimension of varieties}

Given an operad $\mathfrak M$, the number $c_n(\mathfrak M)=\dim\mathfrak M(n)$, $n\ge 1$,
(if it is finite)
is called {\em codimension} of $\mathfrak M$.
The growth of codimensions, namely, of $\sqrt[n]{c_n(\mathfrak M)}$
is intensively studied since the seminal paper \cite{GiamZaicev1998} for associative algebras. 

It follows immediately from definition that for a variety $\di\mathfrak M $ or $\tri\mathfrak M$
the codimension may be explicitly evaluated as a product of $c_n(\Perm)$ or $c_n(\ComTrias)$ 
with $c_n(\mathfrak M)$.

\begin{proposition} For every operad $\mathfrak M$, 
$ c_n(\di\mathfrak M) = nc_n(\mathfrak M)$, 
$c_n(\tri\mathfrak M) = (2^n-1)c_n(\mathfrak M)$.
\end{proposition}

In particular, if $\mathfrak M$ is a variety of Lie algebras of polynomial codimension growth 
then so is the variety $\di\mathfrak M$ of Leibniz algebras.

\subsection{Replication of morphisms of operads}
Let $\mathfrak M,\mathfrak N$ be two operads.
Suppose $\omega : \mathfrak N\to \mathfrak M$ is a morphism of operads.
Then for every algebra $A$ in $\mathfrak M$ one may define $A^{(\omega )}\in \mathfrak N$,
a new algebra structure on the same linear space $A$.

The well-known examples include $-: \Lie \to \As$, $x_1x_2\mapsto x_1x_2-x_2x_1$,
a similar morphism $\Mal \to \Alt$, as well as
$+:\Jord \to \As$, $x_1x_2\mapsto x_1x_2+x_2x_1$, and many others.

For every $B\in \mathfrak N$ there exists unique (up to isomorphism) algebra
$U_\omega (B)\in \mathfrak M$ such that:
\begin{itemize}
 \item There exists a homomorphism $\iota: B\to U_{\omega }(B)^{(\omega )}$ of algebras in $\mathfrak N$;
 \item For every algebra $A\in \mathfrak M$ and for every
 homomorphism $\psi : B\to A^{(\omega )}$ there exists unique homomorphism
$\xi : U_{\omega }(B)\to A$ of algebras in $\mathfrak M$ such that
$\psi(b) = \xi(\iota (b))$ for all $b\in B$.
\end{itemize}
The algebra $U_\omega (B)$ is called the {\em universal enveloping} algebra of $B$
with respect to $\omega $. Note that $\iota $ is not necessarily injective,
e.g., for the Albert algebra $H_3(\mathbb O)\in \Jord $
the universal enveloping associative algebra (with respect to $+$) is equal to $\{0\}$.

Definition \ref{defn:Di-Tri-Algebras} immediately implies

\begin{proposition}
 Given a morphism of operads
$\omega : \mathfrak N\to \mathfrak M$, the map
$\id\otimes \omega : \tri\mathfrak N=\ComTrias\otimes \mathfrak N\to \ComTrias\otimes \mathfrak M = \tri\mathfrak M$
is also a morphism of operads.
\end{proposition}

A similar statement for di-algebra case obviously holds.

\subsection{PBW-type problems}
The following natural problems appear each time when we consider a morphism of
operads $\omega : \mathfrak N\to \mathfrak M$.
\begin{itemize}
 \item Embedding problem: Whether every $B\in \mathfrak N$  is special with respect to $\omega $?
  \item Ado problem: Whether every finite-dimensional algebra $B\in \mathfrak N$ is a subalgebra of
$A^{(\omega )}$, where $A\in \mathfrak M$, $\dim A<\infty $?
 \item Poincar\'e---Birkhoff---Witt (PBW) problem: Given $B\in \mathfrak N$,
what is the structure of the universal enveloping algebra $U_{\omega }(B)\in \mathfrak M$?
\end{itemize}

Suppose $\mathfrak N$ and $\mathfrak M$ are varieties
of $\Sigma-$ and $\Sigma'$-algebras, respectively.
Throughout this section, assume $\nu(f)\ge 2$ for all $f\in \Sigma \cup \Sigma'$.

The following Lemma is an immediate corollary of definitions.

\begin{lemma}
For every morphism of operads
$\omega : \mathfrak N\to \mathfrak M$
and for every
 $A\in \tri\mathfrak M$, $C\in \ComTrias$ we have
\[
C\otimes A^{(\omega )} = (C\otimes A)^{(\id\otimes \omega )} \in \tri\mathfrak N.
\]
\end{lemma}

A similar statement holds for di-algebras \cite{KolVoronin2013}.

\begin{theorem}\label{thm:EmbedAdo_problemsReplicated}
If the embedding problem has positive solution for $\omega: \mathfrak N\to \mathfrak M$
then it has positive solution for $\id\otimes \omega : \di\mathfrak N \to \di\mathfrak M$
and for $\id\otimes \omega : \tri\mathfrak N \to \tri\mathfrak M$.
The same statement holds for the Ado problem.
\end{theorem}

\begin{proof}
 Let us consider tri-algebra case. If $T\in \tri\mathfrak N$
then $\widetilde T\in \mathfrak N$ and $T\subset C_2\otimes \widetilde T$.
If $\widetilde T\subseteq A^{(\omega )}$ for some $A\in \mathfrak M$
then $T\subseteq C_2\otimes \widetilde T\subseteq C_2\otimes A^{(\omega )} = (C_2\otimes A)^{(\id\otimes \omega )}$,
$C_2\otimes A \in \tri\mathfrak M$.

Finally, if $\dim T<\infty$ then $\dim\widetilde T<\infty $ and $\dim(C_2\otimes \widetilde T)<\infty$,
the same holds for $A$. Hence, if $\widetilde T$ has a finite-dimensional envelope then so is $T$.
\end{proof}

Let $T\in \tri\mathfrak N$, $\widetilde T\in \mathfrak N$,
$\iota: \widetilde T \to U_\omega (\widetilde T)^{(\omega )}$ as above.

\begin{theorem}\label{thm:PBW_Replicated}
The subalgebra generated in
$C_2\otimes U_\omega (\widetilde T)\in \tri\mathfrak M$
by the set
$\{e_1\otimes \iota(\bar a) + e_2\otimes \iota(a) \mid a\in T\}$
is isomorphic to
$U_{\id\otimes \omega }(T)\in \tri\mathfrak M$.
\end{theorem}

\begin{proof}
Consider $\id\otimes \iota: C_2\otimes \widetilde T \to C_2\otimes U_\omega (\widetilde T)^{(\omega )}
= (C_2\otimes U_\omega (\widetilde T) )^{(\id\otimes \omega )}$. Hence, the restriction of
$\id\otimes \iota $ to $T\subseteq C_2\otimes \widetilde T$ maps $a\in T$ to
$e_1\otimes \iota(\bar a) +  e_2\otimes \iota (a)$.
Denote by $U$ the subalgebra of $C_2\otimes U_\omega (\widetilde T)\in \tri\mathfrak M$
generated by $(\id\otimes \iota) (T)$. Let us check the universal property for $U$.

Suppose $A\in \tri\mathfrak M$, and $\psi: T \to A^{(\id\otimes \omega )}$.
By Lemma~\ref{lem:Functorial_Tilde}, there exists
$\tilde \psi: \widetilde T \to \widetilde{ A^{(\id\otimes \omega )} }$.
Note that $A^{(\id\otimes \omega )}_0\subseteq A_0$: It follows from the construction
of $\tilde A$, see the proof of Theorem~\ref{thm:EmbeddingIntoCurrent}.
Therefore, there exists natural homomorphism
$\widetilde{A^{(\id\otimes \omega )}} \to \tilde A ^{(\omega )}$, and we may consider
$\tilde \psi $ as a homomorphism from $\widetilde T$ to $\tilde A ^{(\omega )}$,
where
\[
 \tilde\psi(\bar a) =\overline{\psi(a)}\in A/A_0, \quad \tilde \psi(a)=\psi (a)
\]
for $a\in T$.

By definition, there exists a homomorphism of $\mathfrak M$-algebras
$\xi : U_\omega (\widetilde T) \to \tilde A$ such that $\xi(\iota(x))=\tilde \psi(x)$, $x\in \widetilde T$.
Then $\id\otimes \xi: C_2\otimes U_\omega (\widetilde T) \to C_2\otimes \tilde A$
is a homomorphism of $\tri\mathfrak M$-algebras. Moreover, it is easy to see that
$(\id\otimes \xi)((\id\otimes \iota)(T))\subseteq A\subseteq C_2\otimes \tilde A$.
Hence, the restriction of $\id\otimes \xi $ to $U\subseteq C_2\otimes U_\omega (\widetilde T)$
is the desired homomorphism of $\tri\mathfrak M$-algebras $U\to A$.
\end{proof}

The similar statement obviously holds for di-algebras (consider $C_2$ as an algebra in $\Perm $).
For example, the morphism $-: \Lie \to \As$
leads to $\id\otimes -: \Leib \to \di\As$ considered in \cite{LodayPir1993}, see also \cite{Loday01}.
The PBW Theorem for Leibniz algebras is an immediate corollary of Theorem~\ref{thm:PBW_Replicated}.

Let us deduce PBW Theorem for $\tri\Lie $ algebras as an application of Theorem~\ref{thm:PBW_Replicated}.
Every $L\in \tri\Lie $ with operations $[\cdot \vdash \cdot]$, $[\cdot \dashv \cdot]$, and $[\cdot \perp \cdot ]$ 
gives rise to the following Lie algebras: $\bar L=L/L_0$ and $L_\perp = (L, [\cdot \perp \cdot])$.

\begin{corollary}
Let $L\in \tri\Lie $. Then $U_{\id\otimes {-}}(L)\in \tri\As$
as a linear space is isomorphic to 
$U(\bar L)\otimes U_0(L_\perp)$, 
where $U(\cdot )$ is the ordinary universal enveloping associative 
algebra with identity, $U_0(\cdot )$ stands for its augmentation 
ideal.
\end{corollary}

\begin{proof}
Suppose $B\subset L$ is a well-ordered linear basis of $L$.
 It is easy to see that the defining identities of $U_{\id\otimes {-}}(L)$, namely, 
\[
 a\vdash b - b\dashv a = [a\vdash b], \quad a\perp b - b\perp a = [a\perp b], 
\quad a,b\in B,\ a\ge b.
\]
allow to present every element of $U_{\id\otimes {-}}(L)\in \tri\As$ as a linear combination of 
\begin{equation}\label{eq:PBW_normalForm}
u= (a_1\vdash \dots \vdash a_n\vdash b_1)\perp b_2\perp\dots \perp b_m, 
\end{equation}
 $a_i,b_j\in B$, $a_1\le \dots \le a_n$, $b_1\le \dots \le b_m$, $n\ge 0$, $m\ge 1$. 
and $\bar a_1,\dots , \bar a_n$ are linearly independent in $\bar L$

It remains to show that the elements \eqref{eq:PBW_normalForm} are linearly independent 
in  $U_{\id\otimes {-}}(L)$. By Theorem \ref{thm:PBW_Replicated}, 
$U_{\id\otimes {-}}(L)\subseteq C_2\otimes U(\widetilde L)$. Identify $x$ and $\iota (x)$
for $x\in \widetilde L$ and evaluate
\[
\tilde u = (\tilde a_1\vdash \dots \vdash \tilde a_n\vdash \tilde b_1)\perp \tilde b_2\perp\dots \perp \tilde b_m ,
\]
where $\tilde a = e_1\otimes \bar a + e_2\otimes a$, $a\in B$. 
By the definition of $C_2$, 
\[
 \tilde u = e_1\otimes \bar a_1\dots \bar a_n\bar b_1\dots \bar b_m 
 + e_2 \otimes \bar a_1\dots \bar a_n b_1\dots b_m.
\]
By the choice of $a_i$, $b_j$ the second summands are linearly independent in $C_2\otimes U(\widetilde L)$.
\end{proof}

\subsection{Special identities}
Let $\omega : \mathfrak N\to \mathfrak M$ be a morphism of operads.
An algebra $B\in \mathfrak N$ is said to be {\em special} with respect to $\omega $
if there exists $A\in \mathfrak M$ such that $B$ is a subalgebra in $A^{(\omega )}$.

The class of all special algebras in $\mathfrak N$ with respect to $\omega $
may not form a subvariety of $\mathfrak N$: It is not closed with respect to
homomorphic image. The variety generated by
all special algebras is denoted by $\mathrm S^{(\omega )} \mathfrak N$.
The corresponding operad is an image of $\mathfrak N$.
Nonzero elements of the kernel of the
corresponding morphism of operads (if they exist) are exactly all polylinear identities that hold on
all special algebras in $\mathfrak N$ but do not hold on the entire $\mathfrak N$. Such identities
are called special (with respect to $\omega $).

\begin{theorem}\label{thm:SpecialIdentities}
 If $\mathrm{char}\,\Bbbk =0 $
 then the following equation holds for varieties:
 \[
  S^{(\id\otimes \omega) }\tri\mathfrak N = \tri S^{(\omega) }\mathfrak N.
 \]
\end{theorem}

The proof is completely similar to di-algebra case in \cite{KolVoronin2013}.
The only difference appears in using $C_2\in \ComTrias$ instead of
$\Bbbk [x]\in \Perm$, where the $\Perm $-algebra structure on polynomials
was given by $f(x)g(x) = f(0)g(x)$.
Let us sketch the main steps of the proof in these new settings.

\begin{proof}
($\subseteq $)  It is enough to prove that every
 $T\in \tri\mathfrak N$ which is special with respect to $\id\otimes \omega $
 satisfies replicated polylinear special identities. Indeed,
 if $T\subseteq A^{(\id\otimes \omega)}$ for $A\in \tri\mathfrak M$
 then $\psi: \widetilde T \to \widetilde {A^{(\id\otimes \omega)}} \to \tilde A^{(\omega )}$
 is a homomorphism of $\tri\mathfrak N$-algebras.
 Then $\id\otimes \psi: C_2\otimes \widetilde T \to C_2\otimes \tilde A^{(\omega )} $
 is a homomorphism of $\tri\mathfrak N$-algebras which is injective on $T\subseteq C_2\otimes \widetilde T$.
 Hence, $T$ satisfies all identities that hold on $C_2\otimes \tilde A^{(\omega )} \in \tri S^{(\omega )}\mathfrak N$.

($\supseteq $) If $T\in \tri S^{(\omega) }\mathfrak N$ then $\widetilde T\in S^{(\omega) }\mathfrak N$
and thus $\widetilde T$ is a homomorphic image of a special algebra $B\subseteq A^{(\omega)}$,
$A\in \mathfrak M$.
It is straightforward to deduce that $C_2\otimes \widetilde T$ is then a homomorphic image
of a special algebra $(C_2\otimes A)^{(\id\otimes \omega )}$. Therefore,
$T$ belongs to $S^{(\id\otimes \omega) }\tri\mathfrak N$.
\end{proof}

\subsection{TKK construction for tri-Jordan algebras}
The classical Tits---Kantor---Koecher (TKK) construction
of a Lie algebra $T(J)$ for a Jordan algebra $J$
is known to preserve simplicity, nilpotence and strong (Penico) solvability.
Moreover, $T(J)$ is a $\mathbb{Z}_3$-graded Lie algebra
$J^+\oplus S(J)\oplus J^-$, where $J^{\pm}$ are isomorphic copies of the space $J$,
$S(J)$ is the structure algebra constructed by inner derivations and operators of left multiplication in~$J$
\cite{Jacobson}.

The TKK construction for Jordan dialgebras was done in \cite{DiTKK}.
There was also proved an analogue of Zhevlakov theorem \cite{ZhSlShSh} which for ordinary Jordan algebras
states that any finitely generated solvable Jordan algebra is nilpotent.

\begin{proposition}
A finitely generated solvable tri-Jordan algebra is nilpotent.
\end{proposition}

\begin{proof}
Let us consider finitely generated and solvable  $J\in \tri\Jord$.
By Lemma~\ref{lem:Tilde_in_Var}, $\tilde J = \bar J \oplus J \in \Jord$, it
has to be a finitely generated and solvable Jordan algebra by the construction.
By the Zhevlakov theorem, $\tilde J$ is nilpotent. Hence,
$C_2 \otimes \tilde J \in \tri\Jord$ is also nilpotent, and by Theorem \ref{thm:EmbeddingIntoCurrent} so is
$J\subseteq C_2 \otimes \tilde J$.
\end{proof}

The notion of strong solvability for Jordan algebras is translated to di- and tri-algebras
in a straightforward way (the minimal change is due to absence of commutativity).
For a tri-Jordan $J$, the language $\Sigma^{(3)}$ contains three operations 
$\vdash, \dashv, \perp$ (note that $a\vdash b = b\dashv a$).
Consider the sequence
\[
\begin{gathered}
J^{(1)} = J,\quad J^{(2)} = J\cdot J,\\
J^{(n+1)} = J^{(n)}\cdot J^{(n)} + 
J\cdot (J^{(n)}\cdot J^{(n)}) +
(J^{(n)}\cdot J^{(n)})\cdot J,\quad n > 1,
\end{gathered}
\]
where $A\cdot B$ stands for $A\vdash B+A\perp B$.
All $J^{(n)}$ are ideals of $J$. If there exists $N\geq 1$ such that
$J^{(N)} = 0$ then $J$ is said to be {\em strongly solvable} (or Penico solvable).

Let us state an analogue of the TKK construction for tri-Jordan.

\begin{proposition}
For every $J\in \tri\Jord $ there exists  $T(J)\in \tri\Lie$
such that the following properties hold:
\begin{itemize}
\item
 $T(J) = J_{-1}\oplus J_0\oplus J_1$ is $\mathbb{Z}_3$-graded algebra, where
the spaces $J_{-1},J_1$ are copies of $J$;
\item $T(J)$ is nilpotent if and only if $J$ is nilpotent;
\item $T(J)$ is  solvable if and only if $J$ is strongly solvable.
\end{itemize}
\end{proposition}

\begin{proof}
Let us consider $X(J) = C_2\otimes T(\tilde J)$,
where $T(\tilde J)=\tilde J^-\oplus St(\tilde J) \oplus \tilde J^+$ is the TKK construction
for Jordan algebra $\tilde J$,
where
where $St(\tilde J)$ is the structure algebra of $\tilde J$.
By Lemma~\ref{lem:Tilde_in_Var}, $X(\tilde J)\in \tri\Lie$.
We can represent $X(J)$ as a $\mathbb{Z}_3$-graded space
\[
(C_2\otimes\tilde J^-)\oplus (C_2\otimes St(\tilde J))\oplus (C_2\otimes\tilde J^+).
\]
Let $J^{\pm}$ be subspaces in $C_2\otimes\tilde J^{\pm}$ spanned by isomorphic
images of elements
$e_1\otimes \bar a+e_2\otimes a$, $a\in J$.
The subalgebra $T(J)$ generated by $J^+$ and $J^-$ in $X(J)\in \tri\Lie$
is the required one. Indeed,
$T(J)$ is nilpotent or solvable if and only if $J$ is nilpotent or strongly solvable, respectively,
 because of the definitions of $C_2\otimes \tilde J$ and properties of
TKK construction for ordinary algebras.
\end{proof}

\subsection{Tri-Jordan polynomials}
Another classical question is related with Cohn's description of Jordan polynomials in the
free associative algebra \cite{Cohn54}. Suppose $\mathrm{char}\,\Bbbk \ne 2$.
For the morphism of operads $+: \Jord \to \As$ defined by $x_1x_2\mapsto x_1x_2+x_2x_1$,
the free algebra $SJ\langle X\rangle =S^{(+)}\Jord\langle X\rangle $ is a subspace of $\As\langle X\rangle $,
elements of $SJ\langle X\rangle$ are called Jordan polynomials. It is well-known since \cite{Cohn54} that
 $SJ\langle X\rangle  \subseteq H\langle X\rangle $, where  $ H\langle X\rangle$
is the space of symmetric elements with respect to involution $\sigma: x_1\dots x_n\mapsto x_n\dots x_1$;
The embedding is strict if and only if $|X|>3$.

For di-Jordan algebras a similar question was considered in \cite{Voronin2013}: It was shown that
$S^{(+)}\di\Jord\langle X\rangle $ lies in the space of symmetric elements (with respect to naturally
defined involution), and the embedding is strict if and only $|X|>2$. Here we use $(+)$
for $(\id\otimes {+})$ to simplify notations.

Theorem \ref{thm:EmbeddingIntoCurrent} provides a way to solve the same question for tri-Jordan algebras.
Let us sketch the proof which is even simpler than the proof in di-algebra case \cite{Voronin2013}.

Denote by $\sigma $ the linear map $\tri\As\langle X\rangle \to \tri\As\langle X\rangle$
such that
\[
 \begin{gathered}
  \sigma (x)=x, \quad x\in X; \\
 \sigma (u\vdash v) = \sigma(v)\dashv \sigma (u),  \\
 \sigma (u\dashv v) = \sigma(v)\vdash \sigma (u), \\
 \sigma (u\perp v) = \sigma(v)\perp \sigma (u), \quad u,v\in \tri\As\langle X\rangle. \\
 \end{gathered}
\]
Denote $\tri H\langle X\rangle = \{ f\in \tri\As\langle X\rangle \mid \sigma(f)=f \}$.

\begin{proposition}
 For every $X$, $S^{(+)}\tri\Jord\langle X\rangle \subseteq \tri H \langle X\rangle$.
The embedding is strict if and only if $|X|>1$.
\end{proposition}

\begin{proof}
Obviously, $S^{(+)}\tri\Jord\langle X\rangle \subseteq \tri H \langle X\rangle$.
If $|X|>1$ then
\[
f= (x_1\vdash x_2\vdash x_2)\perp x_1 + x_1\perp (x_2\dashv x_2 \dashv x_1) \in \tri H \langle X\rangle \setminus S^{(+)}\tri\Jord\langle X\rangle.
\]
Indeed, consider the tri-algebra analogue of the Grassmann algebra $\wedge \langle \xi_1,\dots, \xi_n\rangle$ constructed as follows.
Associative algebra
\[
 A_n = \As \langle \bar\xi_1,\dots \bar\xi_n, \xi_1,\dots, \xi_n \mid ab=-ba, \, a^2=0,\ a,b\in\{\xi_i,\bar \xi_i\mid i=1,\dots, n\} \rangle
\]
is equipped with homomorphic averaging operator $\tau$ given by $\xi_i \mapsto \bar\xi_i$,
$\bar\xi_i\mapsto \bar\xi_i$. Therefore, $A_n^{(\tau )}\in \tri\As$ by \ref{thm:AveragingOper}.

The epimorphism
$\theta:\tri\As \langle x_1,x_2\rangle \to A_2^{(\tau )}$ defined by
 $x_1\mapsto \xi_1$, $x_2\mapsto \xi_2 $
annihilates $S^{(+)}\tri\Jord\langle x_1,x_2\rangle$, but does not annihilate $f$:
\[
\theta(f)= \bar\xi_1 \bar\xi_2 \xi_2 \xi_1 + \xi_1 \xi_2\bar\xi_2 \bar\xi_1
= 2\bar\xi_1 \bar\xi_2 \xi_2 \xi_1 \ne 0.
\]

If $|X|=1$, $X=\{x\}$, then the equality $S^{(+)}\tri\Jord\langle x\rangle \subseteq \tri H \langle x\rangle$
may be derived from Theorem~\ref{thm:EmbeddingIntoCurrent} and the Cohn Theorem for ordinary algebras.
The involution $\sigma $ of $\tri\As\langle X\rangle $ may be extended to
$\widetilde{\tri\As}\langle X\rangle$ and $C_2\otimes \widetilde{\tri\As}\langle X\rangle$
in the natural ways. Note that $\widetilde{\tri\As}\langle x\rangle$ is a homomorphic image
of $\As\langle \bar x, x\rangle$, and so
$C_2\otimes \As\langle \bar x, x\rangle$ maps onto
$\tri\As\langle x\rangle\subseteq C_2\otimes \widetilde{\tri\As}\langle x\rangle $.
If $\sigma (f)=f$ for $f\in \tri\As\langle x\rangle$
then $f$ has a preimage in $C_2\otimes H\langle \bar x, x\rangle $. The latter coincides with
$C_2\otimes SJ\langle \bar x, x\rangle $ and thus $f$ belongs to $S^{(+)}\tri\Jord \langle x\rangle $.
\end{proof}

\section{Problems on splitted algebras}\label{sec5}
Less is known about relations between operads of pre- and post-algebras that are (in quadratic binary case)
Koszul dual to di- and tri-algebras, respectively. Apart from already considered relations with Rota---Baxter operators,
we may prove analogues of some results from the previous section.

\subsection{Splitting morphisms of operads}
Let us show how a morphism of operads
$\omega : \mathfrak N\to \mathfrak M$
induces a functor on the corresponding varieties of
pre- and post-algebras. We will consider the case
of post-algebras since all constructions for pre-algebras may be obtained
by restriction.

As above, assume $\Sigma $ and $\Sigma'$ are the languages of $\mathfrak M$ and $\mathfrak N$,
respectively.

Let $A\in \post\mathfrak M$.
Define a structure of a $\Sigma^{\prime(3)}$-algebra on the space $A$
as follows. Given $f\in \Sigma'$, $\nu(f)=n$, $H\in \mathcal P(n)$,
we have to define $f^H(a_1,\dots, a_n)$, $a_i\in A$.
Consider $(\ComTrias\langle y_1,y_2,\dots \rangle \boxtimes A)^{(\omega )} \in \mathfrak N$,
and evaluate
\[
f (y_1\otimes a_1,\dots, y_n\otimes a_n)
 =\sum\limits_{H\in \mathcal P(n)} e^{(n)}_H(y_1,\dots, y_n) \otimes b_H.
\]
Here $b_H \in A$ are uniquely defined.
Finally, set
\[
 f^H(a_1,\dots, a_n)=b_H.
\]
Denote the $\Sigma'$-algebra obtained by $A^{(\post\omega)}$.

In a similar way ($|H|=1$), $A^{(\pre\omega)}\in \pre\mathfrak N$
may be defined for $A\in \pre\mathfrak M$.

\begin{proposition}
If
$\omega : \mathfrak N\to \mathfrak M$ is a morphism of operads and
$A\in \post\mathfrak M $ then $A^{(\post\omega)}\in \post\mathfrak N$.
\end{proposition}

\begin{proof}
 Immediately follows from the definition since
\[
(\ComTrias\langle y_1,y_2,\dots \rangle \boxtimes A) ^{(\omega )} =
\ComTrias\langle y_1,y_2,\dots \rangle \boxtimes A^{(\post\omega)}.
\]
\end{proof}

\begin{example}
 Consider the following morphism from the operad of Lie triple systems (LTS) to the operad of Lie algebras:
 \[
 \begin{aligned}
  \omega :{}&  LTS \to \Lie , \\
   & [x_1,x_2,x_3]\mapsto [[x_1,x_2],x_3]
  \end{aligned}
 \]
Then for every $L\in \pre\Lie$ the following operations define $L^{(\pre\omega)}\in \pre LTS$:
\[
\begin{gathered}
 {}[x_1,x_2,x_3]_1 = x_3(x_2 x_1), \quad
 [x_1,x_2,x_3]_2 = -[x_2,x_1,x_3]_1 = -x_3(x_1x_2), \\
 [x_1,x_2,x_3]_3 = (x_1x_2)x_3-(x_2x_1)x_3.
 \end{gathered}
\]
Indeed, consider $P = \Perm\langle y_1,y_2,y_3 \rangle $, and evaluate
\[
 [y_1\otimes x_1,y_2\otimes x_2,y_3\otimes x_3]=[[y_1\otimes x_1, y_2\otimes x_2],y_3\otimes x_3]
\]
in $P\boxtimes \pre\Lie\langle x_1,x_2,x_3\rangle $
assuming $ab = [a,b]_2= -[b,a]_1$ in $\pre\Lie $:
\begin{multline}\nonumber
[[y_1\otimes x_1, y_2\otimes x_2],y_3\otimes x_3]
 =
 [y_1y_2\otimes x_1x_2 - y_2y_1\otimes x_2x_1 , y_3\otimes x_3] \\
 =
 y_1y_2y_3\otimes (x_1x_2)x_3 - y_3y_1y_2 \otimes x_3(x_1x_2)
 - y_2y_1y_3 \otimes (x_2x_1)x_3 + y_3y_2y_1 \otimes x_3(x_2x_1)
\end{multline}
It remains to collect similar terms to get the desired expressions.
\end{example}

\subsection{On the special identities for pre- and post-algebras}
It remains unclear how to solve in general the analogues of PBW-type problems 
for pre- and post-algebras. For special identities, however, we may state a partial 
result and show by example that an analogue of Theorem \ref{thm:SpecialIdentities} 
does not hold.

Given a morphism of operads
$\omega : \mathfrak N\to \mathfrak M$, one may consider the 
induced morphisms $\pre\omega$, $\post\omega$, and define
varieties
$S^{(\pre\omega) }\pre\mathfrak N$ and $S^{(\post\omega)}\post\mathfrak N$
generated by all special algebras in $\pre\mathfrak N$ and $\post\mathfrak N$,
respectively.

\begin{proposition}\label{prop:SpecialPrePost}
 Over a field of zero characteristic, we have the following relations:
 $S^{(\pre\omega) }\pre\mathfrak N \subseteq \pre S^{(\omega)}\mathfrak N$,
 $S^{(\post\omega) }\post\mathfrak N \subseteq \post S^{(\omega)}\mathfrak N$.
\end{proposition}

\begin{proof}
Let us consider the case of post-algebras.
 It is enough to show that every special algebra in $\post\mathfrak N$
 belongs to $\post S^{(\omega)}\mathfrak N$.
 Suppose $T\in S^{(\post\omega) }\post\mathfrak N$,
 $T\subseteq A^{(\post\omega)}$. Fix $C=\ComTrias\langle y_1,y_2,\dots \rangle$
 and note that
 $C\boxtimes T\subseteq C\boxtimes A^{(\post\omega)}=(C\boxtimes A)^{(\omega )}$,
 i.e., $C\boxtimes T \in S^{(\omega )}\mathfrak N$.
 By definition, $T\in \post S^{(\omega )}\mathfrak N$.
\end{proof}

Let us state an example to show that the converse embedding may not hold.
Although the language in the example below contains unary operation,
it is a derivation with respect to the binary product. Hence, Theorem \ref {thm:SpecialIdentities}
for di- or tri-algebras would remain valid in these settings (see Remark \ref{rem:UnitaryOps}).
Thus, the example stated below shows an essential difference between 
di-, tri-algebras  and pre-, post-algebras.

\begin{example}\label{exmp:pre-Counter}
 Let $\mathfrak N = \Perm$, and let $\mathfrak M$ governs the variety
 of associative commutative algebras with a derivation (unary operation) $\partial $
 such that $\partial^2=0$. Consider
 $\omega : \mathfrak N\to \mathfrak M$ given by $x_1x_2 \mapsto \partial(x_1)x_2$.
\end{example}

Here $\Sigma '=\{\cdot \}$, one binary operation; $\Sigma =\{\cdot, \partial \}$.

It is well-known that $\omega $ determines a functor from the variety $\mathfrak M$
to $\mathfrak N=\Perm $ \cite{Loday01}.
Moreover, every algebra of the form $A^{(\omega )}$, $A\in \mathfrak M$,
is 3-nilpotent. Since there are no identities of smaller degree, the variety
$S^{(\omega )}\mathfrak N$ coincides with $N_3$, the variety of algebras satisfying
$x(yz)=(xy)z=0$.

It is straightforward to find the defining identities of $\pre N_3$:
\begin{equation}\label{eq:preS-N3}
\begin{gathered}
(x\prec y)\prec z = 0, \quad (x\succ y)\prec z = 0,\quad
(x\prec y + x\succ y)\succ z = 0, \\
x\prec (y\prec z + y\succ z)=0 ,\quad
x\succ (y\prec z)=0 ,\quad
x\succ (y\succ z)=0 .
 \end{gathered}
\end{equation}
Here $\Sigma^{\prime(3)}=\{\succ, \prec \}$, two binary operations.

On the other hand, $\pre\mathfrak M$ consists of $\Perm $-algebras equipped with
a derivation $\partial $ such that $\partial^2=0$. If $A\in \pre\mathfrak M$ then
the operations on $A^{(\pre\omega )}\in \pre\mathfrak N$ are given by
\[
 a\succ b = \partial (a)b, \quad a\prec b = b\partial (a).
\]
Note that
$a\succ b + b\prec a = \partial(a)b+a\partial(b)=\partial(ab) $,
and $\partial(ab)\succ c = 0$ for all $a,b,c\in A$.
Hence, every algebra in $S^{(\pre\omega )}\pre\mathfrak N$
satisfies an identity
\[
(x\succ y + y\prec x) \succ z =0
\]
which does not follow from \eqref{eq:preS-N3}.

\noindent
\footnotesize{Sobolev Institute of Mathematics\\
         Akad. Koptyug prosp., 4 \\
          630090 Novosibirsk\\
         Russia\\
e-mail: pavelsk@math.nsc.ru, vsevolodgu@mail.ru}

\end{document}